\documentclass{amsart}

\usepackage{latexsym}
\usepackage{amssymb}
\usepackage{amsfonts}
\usepackage{amstext}
\usepackage{hyperref}
\usepackage{amsmath}
\usepackage[all]{xy}
\usepackage{graphicx}

\newcommand{\R}{\mathbb{R}}
\newcommand{\C}{\mathbb{C}}

\newcommand{\N}{\mathbb{N}}

\renewcommand{\P}{\mathbb{P}}

\newcommand{\OO}{\mathcal{O}}

\newcommand{\KK}{\mathcal{K}}
\newcommand{\HH}{\mathcal{H}}
\newcommand{\PP}{\mathcal{P}}

\newcommand{\NN}{\mathcal{N}}

\DeclareMathOperator{\ord}{ord}
\DeclareMathOperator{\mult}{mult}

\DeclareMathOperator{\codim}{codim}

\DeclareMathOperator{\dist}{dist}

\DeclareMathOperator{\reg}{reg}
\DeclareMathOperator{\sing}{sing}
\DeclareMathOperator{\Supp}{Supp}
\DeclareMathOperator{\Int}{Int}

\theoremstyle{plain}
\newtheorem{thm}{Theorem}[section]
\newtheorem{dfn}[thm]{Definition}
\newtheorem{lem}[thm]{Lemma}

\newtheorem{cor}[thm]{Corollary}
\newtheorem{rmk}[thm]{Remark}
\newtheorem{exm}[thm]{Example}
\newcommand{\Xreg}{X_{\text{reg}}}
\newcommand{\Xsing}{X_{\text{sing}}}

\numberwithin{equation}{subsection}

\begin{document}
\title{Lelong numbers on projective varieties}
\author{Manuel Rodrigo Parra}
\maketitle

\begin{abstract}

Given a positive closed (1,1)-current $T$ defined on the regular locus of a projective variety $X$ with bounded mass near the singular part of $X$ and $Y$ an irreducible algebraic subset of $X$, we present uniform estimates for the locus inside $Y$ where the Lelong numbers of $T$ are larger than the generic Lelong number of $T$ along $Y$.\\

\end{abstract}

%%%%%%%%%%%%%%%%%%%%%%%%%%%%%%%%%%%%%%%%%%%%%%%%%%%%%%%%%%%%%%%%%%%%%%%%%%%%%%%%%%%%%%%%%%%%%%%5

\section{Introduction}

The understanding of positive closed currents have become an important task in complex analytic geometry due to the rich properties they have. In particular, positive closed currents generalize analytic sets in a natural way, but unlike varieties, they have good compactness properties and we can also use the vast set of tools arising from pluripotential theory.\\

A key step in the study of positive closed currents is the analysis of their singular locus, namely, the locus where the Lelong numbers are large.\\

In the fundamental paper \cite{SiuAnalytic}, Y. T. Siu proved the Zariski upper semi-continuity of Lelong numbers, i.e. if $T$ is a positive closed current on a compact complex manifold $X$, then for any $c>0$ the upper level set

$$\{x\in X\mid \nu(T,x)\geq c\}\subset X$$ is analytic, where $\nu(T,x)$ denotes the Lelong number of $T$ at $x$. Using this, it is also possible to prove Siu's decomposition theorem, that is, any positive closed $(p,p)$-current $T$ can be written as a (possibly infinite) sum of currents of integration on analytic sets plus a residue. Namely,

$$T=\sum_{j\geq1}\lambda_j[A_j]+R$$ where $A_j$ are irreducible analytic subsets of codimension $p$ in $X$, the constants $\lambda_j>0$ are the (generic) Lelong numbers of $T$ along $A_j$ and $R$ is a positive closed $(p,p)$-current with singular locus of small size. This gives us a quite concrete picture, in codimension $p$, of the structure where the Lelong numbers are big for the current $T$.\\

In order to size how big this singular locus is, various consequences can be derived from Demailly's results obtained in \cite{DemRegCurr}. For instance, if $X$ is a compact K\"ahler manifold and $\{x\in X\mid \nu(T,x)>0\}$ is a countable subset of $X$ (i.e. $\{\nu(T,x)\geq c\}$ is finite for all $c>0$) then

$$\sum_{x\in X}\nu(T,x)^{\dim(X)}\leq\int_X\{T\}^{\dim(X)}<+\infty.$$ Note that the above inequality in the case $X$ a Riemann surface (hence $T$ a positive finite measure) is trivial.\\

In this paper we present a result concerning estimates of the locus of positive Lelong numbers for positive closed currents defined on projective varieties with respect to a fixed irreducible subvariety. More precisely, let $X$ be a (possibly singular) projective variety and let $T$ be a positive closed (1,1)-current on $\Xreg$ with bounded mass near $\Xsing$ (see Definition \ref{boundedmasssingular}) and $Y\subset X$ an irreducible algebraic subset of codimension $l$ in $X$. We want to study the locus inside $Y$ where the Lelong numbers of $T$ are larger than the generic Lelong number of $T$ along $Y$. For every $c>0$ we denote by $E_c(T)$ the Lelong upper level sets of $T$ defined as the analytic subset

$$E_c(T):=\overline{\{x\in \Xreg\mid \nu(T,x)\geq c\}}$$  and by $E_c^Y(T)=E_c(T)\cap Y$ the Lelong upper level sets of $T$ at $Y$. Let $$0\leq\beta_1\leq\beta_2\ldots\leq\beta_{\dim(X)-l+1}$$ be the \emph{jumping numbers of $E_c^T(T)$}, i.e. for every $c\in]\beta_p,\beta_{p+1}]$ the algebraic set $E_c^Y(T)$ has codimension $p$ in $Y$ with at least one component of codimension exactly $p$. Let $\{Z_{p,r}\}_{r\geq1}$ be the countable collection of irreducible components of $\bigcup_{c\in]\beta_p,\beta_{p+1}]}E_c^Y(T)$ of codimension exactly $p$ in $Y$ and denote by $\nu_{p,r}$ the generic Lelong number of $T$ at $Z_{p,r}$. Note that $\beta:=\nu(T,Y)$ the generic Lelong number of $T$ along $Y$ corresponds to $\beta_1$. Then we obtain the main result of this paper

\begin{thm}\label{maintheorem}
With the same notation as above, there exist a positive constant $C$, depending only on the geometry of $X$ and $Y$, such that

$$\sum_{r\geq1}(\nu_{p,r}-\beta)^p\int_{Z_{p,r}}\omega^{k-l-p}\leq C\int_{Xreg}T\wedge\omega^{\dim(X)-1},$$ for all $p=1,\ldots,\dim(X)-l+1,$ where $\omega$ is the Fubini-Study metric of $X$.
\end{thm}

As an example, consider the case where $X$ is the complex projective plane $\P^2$, $Y\subset\P^2$ an irreducible curve and $T$ a positive closed (1,1)-current on $\P^2$ with Lelong number $\beta=\nu(T,Y)$ along $Y$; by Siu's decomposition theorem, it is easy to see that $T-\beta[Y]$ is also a positive closed (1,1)-current. Let's pretend that we can restrict $T-\beta[Y]$ at $Y$, i.e. $T-\beta[Y]$ admits local potentials that are not identically $-\infty$. Then we can study the locus of $T$ inside $Y$ directly. Unfortunately, it is not always possible to restrict currents to a given subvariety, therefore having an estimate as the one of Theorem \ref{maintheorem} is useful even when $Y$ has codimension 1. Moreover, if the codimension of $Y$ in $X$ is bigger than one (say $X=\P^3$ and $Y$ an irreducible curve in $\P^3$) then it is not even possible to use Siu's decomposition theorem since the dimensions do not match. If we take $X$ to be smooth, $Y=X$ (so $\beta=0$) we then reproduce one of the results obtained by J. P. Demailly in \cite{DemNumCrit}; it is important to remark that the results obtained in \cite{DemNumCrit} in the smooth case are more precise, since Demailly managed to bound the singular locus of the given positive closed (1,1)-current together with its absolutely continuous part (with respect to the Lebesgue measure) giving a more precise estimate. This happens to be useful for applications in Algebraic Geometry but for our purposes, which are motivated by holomorphic dynamical systems, our inequality will be enough since this will allow us to prove equidistribution results once we know how to control the singular locus of the critical set of iterates of holomorphic maps. The details of this will appear elsewhere. For more examples and details, see Section 4.\\

In section 2 we present some basic facts on positve closed currents and Lelong numbers needed in the sequel. In section 3 we give a fairly detailed sketch of how to approximate positive closed (1,1)-currents by closed currents with analytic singularities and attenuated Lelong numbers; this technique represents the most important tool in our proof. In section 4 we present a few examples in order to show the usefullness of our result and at the end we give the proof of our main result.\\

\newpage

%%%%%%%%%%%%%%%%%%%%%%%%%%%%%%%%%%%%%%%%%%%%%%%%%%%%%%%%%%%%%%%%%%%%%%%%%%%%%%%%%%%%%%%%%%%%%%%%%%%%%%%%%%%%%%5

\section{Basic facts on positive closed currents}

In this section we introduce the main results concerning positive closed currents and Lelong numbers. The basic reference for this section will be the book \cite{DemBook}, chapter III unless otherwise stated.

\subsection{Lelong numbers}

The main tool we have in order to 'measure' the size of the singular locus of a current are the Lelong numbers. Let $X$ be a compact K\"ahler manifold with K\"ahler form $\omega$ and let $T$ be a positive closed $(p,p)$-current on $X$. By definition, the $(k,k)$-form

$$\sigma_T:=T\wedge\omega^{k-p}$$ is a finite positive measure on $X$.\\

If $x\in X$ and $B(x,r)$ is an Euclidean ball with center $x$ and radius $r>0$, then the function

$$r\mapsto \nu(T,x,r):=\frac{\sigma_T\left(B(x,r)\right)}{\pi^pr^{2p}}$$ is increasing in $r>0$. We define the \emph{Lelong number} $\nu(T,x)$ of $T$ at $x$ as the limit

$$\nu(T,x):=\lim_{r\to0^+}\nu(T,x,r).$$

This limit always exists and $\nu(T,x)$ does not depend on neither the chosen local chart nor $\omega$. The quantity defined above can be seen as a generalization of the multiplicity $\mult_x(Z)$ of a variety $Z$ at $x$. More precisely, if $Z$ is an irreducible analytic subvariety of $X$, then
$$\nu([Z],x)=\mult_x(Z),$$ where $[Z]$ denotes the current of integration along $Z$.\\

A very important feature of Lelong numbers is the upper semicontinuity in both variables, which can be obtained from its definition. But the upper semicontinuity of $\nu(T,\cdot)$ is remarkably stronger since it is not only true in the standard topology but also in the Zariski topology: For every positive closed $(p,p)$-current $T$ and every $c>0$ we denote by $E_c(T)$ the \emph{Lelong upper level set}

$$E_c(T):=\{x\in X\mid \nu(T,x)\geq c\}.$$ A fundamental theorem proved by Siu \cite{SiuAnalytic} states that $E_c(T)$ is always an analytic subset of $X$, hence $\nu(T,\cdot)$ is Zariski upper semicontinuous.\\

Note that by Siu's theorem, given any irreducible analytic subset $V$ of $X$, the quantity

$$\nu(T,V):=\min_{x\in V}\nu(T,x)$$ is equal to $\nu(T,x)$ for $x$ generic, i.e.  for $x$ outside a proper analytic subset of $V$. We define the \emph{Lelong number of $T$ along $V$} as $\nu(T,V)$.\\

As a consequence of Siu's theorem, it is possible to prove the following decomposition formula: If $T$ is a positive closed $(p,p)$-current, then there is a unique decomposition of $T$ as a (possibly infinite) weakly convergent series

$$T=\sum_{j\geq1}\lambda_j[A_j]+R$$ where $[A_j]$ is the current of integration over an irreducible analytic variety $A_j\subset X$ of codimension $p$, $\lambda_j>0$ the generic Lelong numbers of $T$ along $A_j$ and $R$ is a positive closed current such that for every $c>0$, the level set $E_c(R)$ has dimension strictly less than $\dim(X)-p$.\\

This formula (known as Siu's decomposition theorem) states that the singular locus of a positive closed current can be decomposed into a union of analytic subsets plus a residual part with small size.

\subsection{Extensions and intersections of currents}

We state here some known results on positive closed currents that we will need in this paper.\\

A subset $P\subset X$ is said to be \emph{complete pluripolar} if for every $x\in P$ there exist an open neighborhood $U\ni x$ and a plurisubharmonic function $u$ not identically $-\infty$ such that

$$P\cap U=\{z\in U\mid u(z)=-\infty\}.$$ In particular all analytic subsets of $X$ are closed complete pluripolar sets.

\begin{thm}[El Mir]\label{El Mir}
Let $P\subset X$ be a closed complete pluripolar subset and let $T$ be a positive closed current on $X\setminus P$ with bounded mass on a neighborhood of every point of $P$. Then, the trivial extension by zero of $T$ on $X$ is a positive closed current.
\end{thm}

It is well known that for any irreducible analytic subset $A\subset X$, the current of integration $[A_{\reg}]$ has finite mass in a neighborhood of every point of $A_{\sing}$, hence the current of integration $[A]$, meaning its extension by zero through $A_{\sing}$, is a well defined positive closed current on $X$.\\

We finally discuss intersection of currents. Given an open set $\Omega\subset\C^k$, a plurisubharmonic function $\varphi$ and a positive closed current $T$ in $\Omega$, we would like to have a notion of intersection $\frac{\sqrt{-1}}{2\pi}\partial\bar\partial\varphi\wedge T$ on $\Omega$. More precisely, we would like to define

$$\frac{\sqrt{-1}}{2\pi}\partial\bar\partial\varphi\wedge T:=\frac{\sqrt{-1}}{2\pi}\partial\bar\partial(\varphi T).$$ The equation above does not always make sense but it is well defined as long as the sizes of the singular sets involved are not too big; note in particular that it is well defined if $\varphi$ is locally integrable with respect to the trace measure of $T$. We proceed to introduce a more general result concerning intersection of currents.\\

Let $T_1,\ldots,T_q$ be positive closed (1,1)-currents with local potentials $\varphi_1,\ldots,\varphi_q$ respectively. We denote by $L(\varphi_j)$ the \emph{unbounded locus} of $\varphi_j$, namely, the set 

$$L(\varphi_j):=\{x\in X\mid \varphi_j\text{ is not bounded near } x\}.$$ 

\begin{thm}\label{intersection of currents}
Let $\Theta$ be a positive, closed $(k-p,k-p)$-current. Assume that for any choice of indices $j_1<\cdots<j_m$ in $\{1,\ldots,q\}$ the set

$$L(\varphi_{j_1})\cap\cdots\cap L(\varphi_{j_m,})\cap\Supp(\Theta)$$ has $(2p-2m+1)$-Hausdorff measure zero. Then the wedge product $T_1\wedge\cdots\wedge T_q\wedge\Theta$ is well defined. Moreover, the product is weakly continuous with respect to monotone decreasing sequences of plurisubharmonic functions.
\end{thm}

We end this subsection with a useful comparison of Lelong numbers of products of currents: If $T_1$ is a positive closed (1,1)-current and $T_2$ is a positive closed $(p,p)$-current such that the product $T_1\wedge T_2$ (which is given locally by the local potentials of $T_1$) is well defined, then 

\begin{equation}\label{lelongproduct}
\nu(T_1\wedge T_2,x)\geq\nu(T_1,x)\nu(T_2,x)
\end{equation} for every $x\in X$.

\subsection{Currents on singular projective varieties}

We will need to deal with positive closed currents defined on singular varieties. If $X$ is a projective variety and $\iota:X\hookrightarrow\P^N$ an embedding, we will say that $\omega$ is a Fubini-Study form on $X$ if $\omega=\iota^*\omega_{\P^N}|_X$, where $\omega_{\P^N}$ is the Fubini-Study form on $\P^N$. Note that $\omega$ is a positive smooth differential form on $\Xreg$.

\begin{dfn}\label{boundedmasssingular}

If $X$ is a (possibly singular) projective irreducible variety and $T$ is a positive closed $(p,p)$-current defined on $\Xreg$, we will say that $T$ has \emph{bounded mass around $\Xsing$} if there exist an open neighborhood $U$ of $\Xsing$ such that

$$\int_{U\cap \Xreg}T\wedge\omega^{\dim(X)-p}<+\infty.$$ 

\end{dfn}

In a complex manifold, for any given two hermitian forms $\omega$ and $\omega'$ there is always a positive constant $A$ such that $A^{-1}\omega\leq\omega'\leq A\omega$, in particular it is easy to see that above definition does not depend of the embedding of $X$.\\

\newpage

%%%%%%%%%%%%%%%%%%%%%%%%%%%%%%%%%%%%%%%%%%%%%%%%%%%%%%%%%%%%%%%%%%%%%%%%%%%%%%%%%%%%%%%%%%%%%%%55

\section{Approximation of (1,1)-currents}

In this section we will discuss the approximation of (1,1)-currents by currents with analytic singularities. The entire section is based on the work of J. P. Demailly (particularly \cite{DemNumCrit}, \cite{DemRegCurr}). However, we add some details of the proof of Theorems \ref{approximation} and \ref{attenuatedlelongnumbers} since these techniques are crucial for this work and the author believes that they are not very well known.\\

The main ingredients of the approximation are the mean value inequality and the Ohsawa-Takegoshi $L^2$-extension theorem

\begin{thm}[Ohsawa-Takegoshi's $L^2$-Extension Theorem]\label{OT}
Let $X$ be a projective manifold. Then there is a positive line bundle $A\to X$ over $X$ with smooth hermitian metric $h_A$ and a constant $C>0$ such that for every line bundle $L\to X$ provided with a singular hermitian metric $h_L$ and for every $x\in X$ such that $h_L(x)\neq0$, there exist a section $\sigma$ of $L+A$ such that 

$$\|\sigma\|_{h_L\otimes h_A}\leq C|\sigma(x)|.$$
\end{thm}

For a proof of Theorem \ref{OT} see for example \cite{DrorNotes}.

\subsection{Approximation by divisors}

Let $X$ be a projective manifold and let $T$ be a positive closed current representing the first Chern class $c_1(L)$ of a hermitian line bundle $L\to X$. More precisely, we can endow $L\to X$ with a singular hermitian metric $h_L$ and curvature form $\Theta(h_L)$ where

$$T\in c_1(L)=\left\{\Theta(h_L)\right\}.$$ Now, let $A\to X$ be an ample line bundle with smooth hermitian metric $h_A=e^{-\varphi_A}$. Its positive curvature form $\omega:=\frac{\sqrt{-1}}{2\pi}\partial\bar\partial\varphi_A$ endows $X$ with a K\"ahler metric. We can fix a smooth hermitian metric $h$ on $L$, hence we can write $h_L=he^{-2\varphi}$ and $T=\Theta(h_L)=\Theta(h)+\frac{\sqrt{-1}}{2\pi}\partial\bar\partial\varphi$. We endow $mL+A$ with the (singular) metric $h_L^{\otimes m}\otimes h_A$ and we define the (finite dimensional) Hilbert space $\HH_m\subset H^0(X;\OO_X(mL+A))$ as

$$\HH_m:=\left\{\sigma\in H^0(X;\OO_X(mL+A))\mid \|\sigma\|^2_m<+\infty\right\}$$ where the norm $\|\cdot\|_m^2$ is given by

$$\|\sigma\|_m^2:=\int_X h_L^{\otimes m}\otimes h_A(\sigma)dV_{\omega}.$$

We present the following theorem which can be found in \cite{DemNumCrit} (see also \cite{Boucksonthesis}).

\begin{thm}\label{approximation}
Let $X$, $T$, $L\to X$, $\varphi$ and $A\to X$ be as before. Let $\{\sigma_{m,j}\}_{j=1}^{N_m}$ be an orthonormal basis of $\HH_m$ and define

$$\varphi_m(x):=\frac1{2m}\log\left(\sum_{j=1}^{N_m}h^{\otimes m}\otimes h_A(\sigma_{m,j}(x))\right).$$

Then there exist positive constants $C_1$ and $C_2$ independent of $m$ such that, for every $x\in X$ we have:

$$\varphi(x)-\frac{C_1}{m}\leq\varphi_m(x)\leq \sup_{z\in B(x,r)}\varphi(z)+\frac{C_2}{m}+C(x,r),$$ where $C(x,r)$ tends to 0 as $r\to0$.

\end{thm}

\begin{proof}[Proof of Theorem \ref{approximation}]

First we cover $X$ by finitely many small open balls $\{B\}$ giving local trivializations for both line bundles $A$ and $L$. On $L,\,A|_B\simeq B\times\C\subset\C^k\times\C$ we pick smooth metrics $\psi$, $\psi_A$ for $L$ and $A$ respectively, i.e. for all $(x,v)\in B\times\C$

$$h(x,v)=|v|^2e^{-2\psi(x)},\quad h_A(x,v)=|v|^2e^{-2\psi_A(x)},$$ hence if $\sigma\in\HH_m$ is a section supported on $B$ we have that

$$h^{\otimes m}\otimes h_A(\sigma(x))=|\sigma(x)|^2e^{-2m\psi(x)-2\psi_A(x)}.$$

Since $\sigma:B\to\C$ is holomorphic, by the mean value inequality for all $x\in B$ and $r<\dist(x,\partial B)$ we have

$$|\sigma(x)|^2\leq\frac{k!}{\pi^k r^{2k}}\int_{B(x,r)}|\sigma(z)|^2dV(z),$$ implying

\begin{multline}
h^{\otimes m}\otimes h_A(\sigma(x))\leq\frac{C}{r^{2k}}e^{-2m\psi(x)-2\psi_A(x)}\int_{B(x,r)}|\sigma(z)|^2dV(z)\leq\\\leq \frac{C}{r^{2k}}e^{2m[\sup_{B(x,r)}\psi-\psi(x)]+2[\sup_{B(x,r)}\psi_A-\psi_A(x)]}\int_{B(x,r)}h^{\otimes m}\otimes h_A(\sigma(z))dV(z).
\end{multline}
Denote by $c(x,r):=\sup_{B(x,r)}\psi-\psi(x)$, $c_A(x,r):=\sup_{B(x,r)}\psi_A-\psi_A(x)$ and note that

$$\int_{B(x,r)}h^{\otimes m}\otimes h_A(\sigma(z))dV(z)\leq\left(\sup_{B(x,r)}e^{2m\varphi}\right)\|\sigma\|^2_m,$$ therefore

\begin{equation}\label{eq1}
h^{\otimes m}\otimes h_A(\sigma(x))\leq \frac{C}{r^{2k}}e^{mc(x,r)+c_A(x,r)}\left(\sup_{B(x,r)}e^{2m\varphi}\right)\|\sigma\|^2_m.
\end{equation}

We can write $\varphi_m$ as

$$e^{2m\varphi_m(x)}=\sup_{\|\sigma\|=1}h^{\otimes m}\otimes h_A(\sigma(x));$$ therefore taking $\log$ of \eqref{eq1} and the supremum over $\|\sigma\|_m=1$ we obtain

\begin{multline}
2m\varphi_m(x)\leq\log\left(\frac{C}{r^{2k}}\right)+mc(x,r)+c_A(x,r)+2m\sup_{B(x,r)}\varphi\Longrightarrow\\\Longrightarrow \varphi_m(x)\leq \sup_{B(x,r)}\varphi+C(x,r)+\frac1{m}\log\left(\frac{C'}{r^{k}}\right),
\end{multline} where $C'>0$ and $C(x,r)\to0$ as $r\to0$.\\

For the other inequality we use Ohsawa-Takegoshi's $L^2$ Extension Theorem: Let $x\in X$ such that $h_L(x)\neq0$. Since $h_L=he^{-2\varphi}$ we can find a section $\sigma\in\HH_m$ such that

$$\|\sigma\|^2_m\leq C^2|\sigma(x)|^2=C^2h^{\otimes m}\otimes h_A(\sigma(x))e^{-2m\varphi(x)},$$ for some $C>0$. Using (again) that

$$e^{2m\varphi_m(x)}=\sup_{\|\sigma\|_m=1}h^{\otimes m}\otimes h_A(\sigma(x)),$$ we take $\log$ of the inequality and the supremum over $\|\sigma\|_m=1$ obtaining

$$\varphi(x)\leq\varphi_m(x)+\frac{C'}{m}.$$ This concludes the proof.

\end{proof} 

As a consequence of the theorem above, we have obtained the following corollary

\begin{cor}\label{analyticapproximation}
Let $X$ be a projective complex manifold and let $T$ be a positive closed (1,1)-current in the cohomology class of a line bundle. Then there exist a sequence of closed (1,1)-currents $T_m$ in the cohomology class of $T$ such that

\begin{itemize}
\item[(i)] $T_m\geq-\frac{1}{m}\omega$;

\item[(ii)] The sequence $T_m$ converges weakly to $T$;

\item[(iii)] For every $x\in X$ the Lelong numbers at $x$ satisfy

$$\nu(T,x)-\frac{C}{m}\leq\nu(T_m,x)\leq\nu(T,x),$$ for some $C>0$. In particular, the Lelong numbers $\nu(T_m,x)$ converge uniformly to $\nu(T,x)$.
\end{itemize}
\end{cor}

\begin{proof}
Let $L\to X$ be a positive hermitian line bundle with singular hermitian metric $h_L$ such that $T\in\{\Theta(h_L)\}$. We can take a smooth metric $h$ on $L$ such that $h_L$ can be written as $h_L=he^{-2\varphi}$ and therefore we can define

$$T_m:=\Theta(h)+\frac{\sqrt{-1}}{2\pi}\partial\bar\partial\varphi_m,$$ with $\varphi_m$ as in the theorem above. Now it is routine to check that the sequence $T_m$ converges to $T=\Theta(h)+\frac{\sqrt{-1}}{2\pi}\partial\bar\partial\varphi$ and that it has the desired properties.
\end{proof}

\subsection{Attenuation of Lelong numbers}

We finish this section with a refined version of the theorem of the subsection above which will allow us to approximate positive closed currents by currents with analytic singularities and attenuated Lelong numbers. We state the main theorem of this section proved in \cite{DemNumCrit}.

\begin{thm}\label{attenuatedlelongnumbers}
Let $X$ be a projective manifold and let $T$ be a positive closed (1,1)-current representing the class $c_1(L)$ of some hermitian line bundle $L\to X$. Fix a sufficiently positive line bundle $G$ over $X$ such that $TX\otimes G$ is nef. Then for every $c>0$ there exist a sequence of closed (1,1)-currents $T_{c,m}$ converging weakly to $T$ over $X$ such that

\begin{itemize}
\item $T_{c,m}\geq-\frac2{m}\omega-cu$, where $u$ is the curvature form of $G$ and;\\

\item $\max\left(\nu(T,x)-c-\dim(X)/m,0\right)\leq\nu(T_{c,m},x)\leq \max\left(\nu(T,x)-c,0\right).$
\end{itemize}
\end{thm}

The proof of the above theorem in a more general case, namely $X$ is a compact K\"ahler manifold and $T$ is any almost positive closed (1,1)-current can be found in \cite{DemRegCurr}; the proof involves a very technical gluing procedure which is beyond the scope of what we want to present here. For the case $X$ projective and $T$ the curvature current of a positive line bundle, the proof is simpler and can be obtained in a more direct way; we present the proof given in \cite{DemNumCrit} with some details added.\\

\begin{proof}

As in Theorem \ref{approximation} it is possible to construct sections $\sigma_{m,j}\in H^0(X;mL+A)$, $1\leq j\leq N_m$ such that

$$\nu(T,x)-\frac{\dim X}{m}\leq\frac1{m}\min_{j=1,\ldots,N_m}\ord_x(\sigma_{m,j})\leq\nu(T,x).$$

We consider the $l$-jet sections $J^l\sigma_{m,j}$ with values in the vector bundle $J^{l}\OO_X(mL+A)$. We have the exact sequence

$$0\to S^lT^*X\otimes\OO_X(mL+A)\to J^l\OO_X(mL+A)\to J^{l-1}\OO_X(mL+A)\to 0.$$ Dualizing the above sequence we obtain the short exact sequence

$$0\to (J^{l-1}\OO_X(mL+A))^*\to (J^l\OO_X(mL+A))^*\to (S^lT^*X\otimes\OO_X(mL+A))^* \to 0$$ which can be rewritten as

$$0\to (J^{l-1}\OO_X(mL+A))^*\to (J^l\OO_X(mL+A))^*\to S^lTX\otimes\OO_X(-mL-A)\to0.$$

Twisting this exact sequence with $\OO_X(mL+2A+lG)$ we obtain that

\begin{multline}\label{shortamplesequence}
0\to (J^{l-1}\OO_X(mL+A))^*\otimes \OO_X(mL+2A+lG)\to \\ \to (J^l\OO_X(mL+A))^*\otimes \OO_X(mL+2A+lG)\to S^lTX\otimes\OO_X(lG+A)\to 0
\end{multline} is exact. By hypothesis, the vector bundle $TX\otimes G$ is nef and therefore $S^l(TX\otimes\OO_X(G))=S^lTX\otimes\OO_X(lG)$ is nef for all symmetric powers of order $l$, hence 

$$S^lTX\otimes\OO_X(lG+A)=\underbrace{\left(S^lTX\otimes\OO_X(lG)\right)}_{\text{nef}}\otimes\underbrace{\OO_X(A)}_{\text{ample}}$$ is ample. Since hte extremes of the exact sequence \eqref{shortamplesequence} are ample, we use induction on $l\geq1$ to conclude that the middle term $$(J^l\OO_X(mL+A))^*\otimes\OO_X(mL+2A+lG)$$ is also ample.\\

By definition of amplitude of vector bundles there exist $q\geq1$ such that

$$S^q\left(J^l\OO_X(mL+A)\right)^*\otimes \OO_X(qmL+2qA+qlG)$$ is generated by holomorphic sections $g_{m,i}$. Using this together with the pairing of $(J^l\OO_X(mL+A))^*$ and $J^l\OO_X(mL+A)$ we obtain sections

$$S^q(J^l\sigma_{m,j})g_{m,i}\in H^0(X;\OO_X(qmL+2qA+qlG))$$ which in a trivialization give us the metric

$$\varphi_{m,l}:=\frac1{qm}\log\sum_{i,j}|S^q(J^l\sigma_{m,j})g_{m,i}|-\frac2{m}\psi_A-\frac{l}{m}\psi_G.$$ Note that $\psi_A$ and $\psi_G$ are smooth; therefore we have

$$\nu(\varphi_{m,l},x)=\frac1{m}\min_{j}\ord_x(J^l\sigma_{m,j})=\frac1{m}\left(\min_j\ord_x(\sigma_{m,j})-l\right).$$ This gives us the inequality

$$\max\left(\nu(T,x)-\frac{l+\dim(X)}{m},0\right)\leq\nu(\varphi_{m,l},x)\leq\max\left(\nu(T,x)-\frac{l}{m},0\right).$$ Finally, for every $c>0$ and every $m\gg0$ it is possible to find $l>0$ such that $c<l/m<c+1/m$, hence

$$\frac{\sqrt{-1}}{\pi}\partial\bar\partial\varphi_{c,m}\geq-\frac2{m}\omega-cu,$$ where $\varphi_{c,m}:=\varphi_{m,l}$ for this choice of $m$, $l$, and $\omega$ and $u$ are the curvature forms of $A$ and $G$ respectively.\\

Now, for any smooth metric $h$ on $L$, the sequence of currents $$T_{c,m}:=\Theta(h)+\frac{\sqrt{-1}}{\pi}\partial\bar\partial\varphi_{c,m},$$ converges weakly to $T$ and satisfies the desired properties.

\end{proof}

\newpage

\section{Some examples}

In this section we provide a few examples showing the value of our result.\\

\begin{exm}
Let's start with the trivial case where $X=Y$ is a projective curve. In this case, the current $T$ is a positive finite measure on $X$, giving us

$$\int_XT\geq\sum_{x\in X}\nu(T,x)=\sum_{r\geq1}\nu_{1,r}.$$ So Theorem \ref{maintheorem} follows immediately.\\

\end{exm}

\begin{exm}\label{reproduceDemailly}
If $X$ is a projective manifold (i.e. smooth) and $Y=X$, hence $l=0$, $\beta=0$ and $E^Y_c(T)=E_c(T)$, then the inequality of Theorem \ref{maintheorem} can be written as $$\sum_{r\geq 1}\nu_{p,r}^p\int_{Z_{p,r}}\omega^p\leq C.$$

In \cite{DemRegCurr} Theorem 7.1, J.P. Demailly proved that under the same assumptions as above, we have that there is a positive constant $C'>0$ such that

$$\sum_{r\geq 1}(\nu_{p,r}-\beta_1)\cdots(\nu_{p,r}-\beta_p)\int_{Z_{p,r}}\omega^p\leq C',$$ where $\beta_1\leq\ldots\leq\beta_{k+1}$ are the jumping numbers of $T$. Observing that $$\nu_{p,r}\geq\nu_{p,r}-\beta_j\quad\forall\,\,j=1,\ldots,p$$ Theorem \ref{maintheorem} shows that

$$\sum_{r\geq 1}(\nu_{p,r}-\beta_1)\cdots(\nu_{p,r}-\beta_p)\int_{Z_{p,r}}\omega^p\leq\sum_{r\geq 1}\nu_{p,r}^p\int_{Z_{p,r}}\omega^p\leq C$$ implying Demailly's result.\\

In the same setting, it is also interesting to observe the two extreme cases $p=1$ and $p=k$:\\

\begin{itemize}
\item[a)] The case $p=1$ follows immediately from Siu's decomposition theorem, since

$$T=\sum_{j\geq1}\lambda_j[A_j]+R$$ with $R\geq0$ and $\{Z_{1,r}\}_r\subset\{A_j\}_j$, therefore

$$T\geq\sum_{r\geq1}\nu_{1,r}[Z_{1,r}]$$ and the result follows after integrating this inequality with $\int\cdot\wedge\omega^{k-1}$.\\

\item[b)] The case $p=k$ is especially interesting when $E_c(T)$ is countable for all $c>0$. A remarkable result proved in Corollary 6.4 of \cite{DemRegCurr} is that if $E_c(T)$ countable then the class $\{T\}$ is nef. Moreover, Demailly gave the more refined inequality

$$\sum_{r\geq1}\nu_{k,r}^k+\int_XT^k_{\text{ac}}\leq\int_X\{T\}^k,$$ where $T_{\text{ac}}$ is the absolutely continuous part in the Lebesgue decomposition of the coefficients of $T$.\\
\end{itemize}
\end{exm}

\begin{exm}
Let $X$ be the projective plane $\P^2$, let $Y\subset\P^2$ be an irreducible curve and let $T$ be the current of integration defined by $T:=(\deg(D))^{-1}[D]$ (hence $\|T\|=1$), where $D\neq Y$ is another irreducible curve. It is not hard to see from the proof of our theorem (see Subsection \ref{thebigproof}) that the constant $C>0$ satisfies $C=\deg(Y)$ and that

$$\sum_{r\geq1}\nu_{1,r}\int_{Z_{1,r}}\omega=\#(Y\cap D)\cdot(\deg(D))^{-1}.$$ Therefore Theorem \ref{maintheorem} is nothing but B{\'e}zout's theorem for these two curves.\\
\end{exm}

\begin{exm}
If $X$ is a projective manifold (i.e. smooth) and $Y\subset X$ an irreducible smooth hypersurface, hence $l=1$, then using Siu's decomposition theorem it is easy to see that the closed (1,1)-current $T-\beta[Y]$ is positive. Assume that $T-\beta[Y]$ admits local potentials not identically $-\infty$ along $Y$. Then we can restrict $T-\beta[Y]$ to $Y$ and the statement is reduced to Example \ref{reproduceDemailly}.\\
\end{exm}

It is important to remark that it is not always possible to restrict positive close currents. Part of the idea in the proof of Theorem \ref{maintheorem} is to restrict an approximation of the current.\\

\begin{rmk}
If $Y\subset X$ has codimension $l>1$ in $X$, then it is not even possible to subtract $[Y]$ from $T$ since the dimensions do not match. Therefore there is no direct method for studying the Lelong numbers of $T$ inside $Y$, so our theorem proves to be useful in the general case.\\
\end{rmk}

%%%%%%%%%%%%%%%%%%%%%%%%%%%%%%%%%%%%%%%%%%%%%%%%%%%%%%%%%%%%%%%%%%%%%%%%%%%%

\newpage 

\section{Proof of the main theorem}\label{thebigproof}

We will divide the proof of Theorem \ref{maintheorem} into three steps. In the first step we will assume that our projective variety $X$ is smooth and that the cohomology class of $T$ is nef. We can then find suitable smooth representatives of $\{T\}$ which, together with the sequence obtained in Theorem \ref{attenuatedlelongnumbers}, will allow us to approximate $T$ by a sequence of currents $T_{c,m}$ of bounded potentials and therefore we will be able to intersect such sequence with the current of integration $[Y]$; this procedure will be the key for obtaining our result in this setting. In the second step, we still assume $X$ smooth but $\{T\}$ not necessarily nef; using that $H^{1,1}(X;\R)$ is finite dimensional and the upper semicontinuity of Lelong numbers, we will replace $T$ by a current $\hat T$ with nef class $\{\hat T\}$ but the same Lelong numbers as $T$ everywhere; then we apply our result in Step 1 to $\hat T$ implying the same conclusion to $T$. Finally, in Step 3 we prove the general case when $X$ is a projective variety, not necessarily smooth, by taking a resolution of singularities of $X$, and applying Step 2 to the strict transform of $T$, $Y$ and $Z_{p,r}$.\\

We recall the notions of \emph{numerically effective (nef), pseudoeffective (psef) and K\"ahler cones} (for details see \cite{MR2095471}, also \cite{Boucksonthesis}). The space of classes of \emph{real (1,1)-forms} $H^{1,1}(X;\R)$ is defined as

$$H^{1,1}(X;\R):=H^{1,1}_{\bar\partial}(X;\C)\cap H^2(X;\R)=\left\{\alpha\in H^{1,1}_{\bar\partial}(X;\C)\mid \bar\alpha=\alpha\right\},$$ where $H^{1,1}_{\bar\partial}(X;\C)$ is the Dolbeault (1,1)-cohomology of $X$.\\

The \emph{K\"ahler cone} $\KK(X)$, the \emph{Psef cone} $\PP(X)$ and the \emph{Nef cone} $\NN(X)$ are defined as

$$\KK(X):=\{\alpha\in H^{1,1}(X;\R)\mid \alpha \text{ can be represented by a K\"ahler form}\},$$ 

\begin{multline*}
\PP(X):=\{\alpha\in H^{1,1}(X;\R)\mid \alpha \text{ can be represented by}\\ \text{a positive closed (1,1)-current}\},\end{multline*}

and

\begin{multline*}
\NN(X):=\{\alpha\in H^{1,1}(X;\R)\mid \text{ if for every } \epsilon>0,\, \alpha \text{ can be represented by a }\\ \text{smooth form } \alpha_{\epsilon} \text{ such that } \alpha_{\epsilon}\geq-\epsilon\omega\}
\end{multline*} respectively. Note that if $X$ is K\"ahler (or projective) the set $\KK(X)$ is not empty.\\

It follows from the definitions above that

$$\emptyset\neq\KK(X)\subset\NN(X)\subset\PP(X)\quad\text{and}\quad\KK(X)=\Int(\NN(X)).$$

We prove now our Main Theorem:\\

{\bf Step 1}: Assume $X$ to be a (smooth) complex projective manifold and the class $\{T\}$ to be nef. For this case, we will actually prove a slightly more general result, where we will be able to 'kill Lelong numbers' locally. More precisely, given any subset $\Xi$ of $Y$ and $p=1,\ldots,\dim(X)-l+1$, we denote the \emph{jumping numbers $b_p=b_p(T,\Xi)$ of $E_c^Y(T)$ with respect to $\Xi$} as

$$b_p:=\inf\{c>0\mid \codim_x(E_c^Y(T);Y)\geq p,\,\,\forall\,x\in \Xi\}.$$

In our situation, the subset $\Xi$ will be a Zariski dense subset of $Y$ with a prescribed geometrical condition, namely, $\Xi$ will be the complement of all irreducible components of $E_c^Y(T)$ of codimension strictly less than $p$. Following Demailly we prove the following lemma

\begin{lem}
Let $\Xi$ be any subset of $Y$ and $0\leq b_1\leq b_2\leq\ldots\leq b_{\dim(X)-l+1}$ the jumping numbers of $E_c^Y(T)$ with respect to $\Xi$. Fix a positive line bundle with smooth curvature $u$ as in Theorem \ref{attenuatedlelongnumbers} and assume that the class $\{T\}$ is nef. Then for every $p=1,\ldots,\dim(X)-l+1$ there exists a positive closed $(l+p,l+p)$-current $\Theta_p$ in $X$ with support on $Y$ such that

\begin{equation}
\{\Theta_p\}=\{Y\}\cdot(\{T\}+b_1\{u\})\cdots(\{T\}+b_p\{u\})\in H^{l+p,l+p}(X;\R),
\end{equation}
\begin{equation}
\Theta_p\geq \sum_{r\geq1}(\nu_{p,r}-b_1)\cdots(\nu_{p,r}-b_p)[Z_{p,r}].
\end{equation}
\end{lem}

\begin{proof}

Let $c>b_1$ and let $\alpha\in\{T\}$ be a smooth real (1,1)-form. Take the sequence of currents $T_{c,m}=\alpha+\frac{\sqrt{-1}}{2\pi}\partial\bar\partial\varphi_{c,m}$ as in Theorem \ref{attenuatedlelongnumbers} where $\varphi_{c,m}$ is singular along $E_c(T)$ and $T_{c,m}\geq-\frac2{m}\omega-cu$. Since we are assuming $\{T\}$ to be nef, for every $m\in\N$ we can pick $\alpha_m\in\{T\}$ smooth such that $\alpha_m\geq-\frac2{m}\omega$ and we can write $\alpha_m$ as $\alpha_m=\alpha+\frac{\sqrt{-1}}{2\pi}\partial\bar\partial\psi_m$ with $\psi_m$ smooth. Set

$$\varphi_{c,m,L}:=\max\{\varphi_{c,m},\psi_m-L\},$$ for $L\gg0$ and $T_{c,m,L}:=\alpha+\frac{\sqrt{-1}}{2\pi}\partial\bar\partial\varphi_{c,m,L}$. Observe that by adding the local potentials of $\omega$ and $cu$ to the the max between $\varphi_{c,m}$ and $\psi_m-L$ we easily conclude that the closed current $T_{c,m,L}$ satisfies

$$T_{c,m,L}+\frac2{m}\omega+cu\geq0.$$

The family of potentials $\{\varphi_{c,m,L}\}$ is bounded everywhere, therefore

$$\Theta_{1,c,m,L}:=[Y]\wedge\left(T_{c,m,L}+\frac2{m}\omega+cu\right)$$ is a well defined positive closed $(l+1,l+1)$-current on $X$ with support on $Y$ by Theorem \ref{intersection of currents}. By extracting a weak limit we define

$$\Theta_{1,c,m}:=\lim_{L\to+\infty}\Theta_{1,c,m,L}$$ on $X$. Since the potentials $\varphi_{c,m,L}$ decrease monotonically to $\varphi_{c,m}$ as $L\to+\infty$ we have that

$$\Theta_{1,c,m}=[Y]\wedge\left(T_{c,m}+\frac2{m}\omega+cu\right)$$ in a neighborhood of $\Xi$, since for every point $x\in\Xi$ we can find a neighborhood $U$ of $x$ such that the unbounded locus of $[Y]$ and $T_{c,m}$ has codimension $\geq l+1$ in $U$ (or real dimension $\leq2\dim(X)-2l-2$) hence by Theorem \ref{intersection of currents} the current $\Theta_{1,c,m}$ is well defined in a neighborhood of $\Xi$, and $\{\Theta_{1,c,m}\}=\{Y\}\cdot\left(\{T\}+\frac2{m}\{\omega\}+c\{u\}\right)$ for every $m\geq1$ and every $c>b_1$ and for every $x\in X$,

$$\nu(\Theta_{1,c,m},x)\geq\ord_x(Y)\nu(T_{c,m},x)\geq\ord_x(Y)(\max\{\nu(T,x)-c-\dim(X)/m,0\}).$$

Note also that the total mass of the family $\{\Theta_{1,c,m}\}$ is uniformly bounded. We extract (modulo a subsequence) a limit

$$\Theta_1:=\lim_{c\searrow b_1}\lim_{m\nearrow+\infty}\Theta_{1,c,m},$$ which satisfies $\{\Theta_1\}=\{Y\}\cdot(\{T\}+b_1\{u\})$ and by the upper semicontinuity of Lelong numbers we obtain

$$\nu(\Theta_1,x)\geq(\nu_{1,r}-b_1),\quad\forall\,x\in Z_{1,r}\,\,\forall\,r\geq1.$$

By Siu's decomposition theorem, $\Theta_1$ can be written as 

$$\Theta_1=\sum_{j\geq1}\lambda_j[V_j]+R_1,$$ where for every $j\geq1$, $V_j$ is an irreducible variety of codimension $l+1$ in $X$, $\lambda_j$ is the generic Lelong number of $\Theta_1$ along $V_j$ and $R_1$ is a positive closed current with upper level sets $E_c(R_1)$ of codimension strictly bigger than $l+1$ for all $c>0$. This in particular implies that for all $r\geq1$ we have that $Z_{1,r}=V_{j_r}$ for some $j_r$ and for a generically chosen $x\in Z_{1,r}$ we obtain

$$\lambda_{j_r}=\nu(\Theta_1,x)\geq(\nu_{1,r}-b_1)\Longrightarrow\Theta_1\geq\sum_{r\geq1}(\nu_{1,r}-b_1)[Z_{1,r}].$$

Now we proceed by induction on $2\leq p\leq\dim(X)-l+1$. We assume we have constructed $\Theta_{p-1}$ with the desired properties and in the exact same way as before, for $c>b_p$ we define the positive closed $(l+p,l+p)$-current

$$\Theta_{p,c,m,L}:=\Theta_{p-1}\wedge\left(T_{c,m,L}+\frac2{m}\omega+cu\right),$$ which is well defined everywhere. The current

$$\Theta_{p,c,m}:=\lim_{L\to+\infty}\Theta_{p,c,m,L}$$ satisfies

\begin{itemize}
\item $\Theta_{p,c,m}=[Y]\wedge\left(T_{c,m}+\frac2{m}\omega+cu\right)$ in a neighborhood of $\Xi$,\\

\item $\{\Theta_{p,c,m}\}=\{Y\}\cdot\left(\{T\}+\frac2{m}\{\omega\}+c\{u\}\right)$ for every $m\geq1$ and every $c>b_p$ and,\\

\item $\nu(\Theta_{p,c,m},x)\geq\nu(\Theta_{p-1},x)\max\{(\nu(T,x)-c-\dim(X)/m),0\}$ for every $x\in X$.\\
\end{itemize}

We extract a weak limit (modulo a subsequence)

$$\Theta_p:=\lim_{c\searrow b_p}\lim_{m\nearrow+\infty}\Theta_{p,c,m},$$ which (by the same arguments as above) satisfies the desired properties.

\end{proof}

{\bf Step 2}: Now assume that $X$ is a complex projective manifold and let $\omega$ be any K\"ahler form on $X$. However, the class $\{T\}$ is not necessarily nef.\\

Let 

\begin{equation}\label{coneone}
\PP_1:=\{\alpha\in \PP(X)\mid \|\alpha\|=1\}\subset H^{1,1}(X;\R)
\end{equation}
be a slice of the pseudoeffective cone of $X$, where $\|\cdot\|$ is any norm on the finite dimensional real vector space $H^{1,1}(X;\R)$. Since $\PP_1$ is compact and $\Int(\NN(X))=\KK(X)\neq\emptyset$ we can pick $A_0=A_0(\{\omega\})>0$ such that $A\{\omega\}+\alpha$ is nef for every $A\geq A_0$ and every $\alpha\in\PP_1$. Note also that the set of of positive closed currents $T$ with a fixed cohomology class is also (weakly) compact. Moreover, by the upper semicontinuity in both variables of the Lelong numbers is easy to see that there exists a constant $\tau=\tau(X)$ such that $\nu(T,x)\leq\tau$ for every $x\in X$ and every positive closed (1,1)-current $T$ so that $\{T\}\in\PP_1$.\\

Now, fixing $A\geq A_0$ we define the positive closed (1,1)-current $\hat{T}:=T+A\omega$. It satisfies:\\

\begin{itemize}
\item $\{\hat T\}\in \NN(X)$,\\

\item $\nu(\hat T,x)=\nu(T,x)$ for every $x\in X$ (in particular, the Lelong upper level sets $E_c^Y(\hat T)$ and $E_c^Y(T)$ coincide, giving us the same decomposition in terms of jumping numbers).\\
\end{itemize}

Taking $\beta=\nu(T,Y)=\nu(\hat T,Y)$ and defining the set 

$$\Xi_p:=\complement\left(\cup_{c>\beta}(\text{Irreducible components of } E_c^Y(T) \text{ of codimension } <p)\right),$$ we obtain that the jumping numbers with respect to $\Xi_p$ satisfy $$b_1(T,\Xi_p)=\ldots=b_p(T,\Xi_p)=\beta.$$ 

If $\{Z_{p,r}\}_{r\geq1}$ are the irreducible components of $E_c^Y(\hat T)=E_c^Y(T)$ for $c\in]\beta_p,\beta_{p+1}]$ of codimension exactly $p$ in $Y$ and $\nu_{p,r}$ the generic Lelong numbers, we apply the previous lemma to $\hat T$, hence we obtain a positive closed $(l+p,l+p)$-current $\Theta_p$ on $X$ with support on $Y$ such that 

\begin{multline*}
\{\Theta_p\}=\{Y\}\cdot(\{\hat T\}+b_1\{u\})\cdots(\{\hat T\}+b_p\{u\})=\\=\{Y\}\cdot(\{T\}+A\{\omega\}+b_1\{u\})\cdots(\{T\}+A\{\omega\}+b_p\{u\})
\end{multline*}
and

$$\sum_{r\geq1}(\nu_{p,r}-\beta)^p[Z_{p,r}]\leq\Theta_p.$$

We apply $\int_X\cdot\wedge\omega^{\dim(X)-l-p}$ to the inequality, giving us

\begin{multline*}
\sum_{r\geq1}(\nu_{p,r}-\beta)^p\int_{Z_{p,r}}\omega^{\dim(X)-l-p}\leq\int_X\Theta_p\wedge\omega^{\dim(X)-l-p}=\\=\int_X[Y]\wedge(T+A\omega+b_1u)\wedge\cdots\wedge(T+A\omega+b_pu)\wedge\omega^{\dim(X)-l-p}\leq\\\leq\int_X[Y]\wedge\left((1+A)\omega+\tau u\right)^p\wedge\omega^{\dim(X)-l-p}=:C.
\end{multline*}

{\bf Step 3}: We now prove the theorem in the general case.\\

Let $\pi:\tilde{X}\to X$ be a resolution of singularities. Since $Y$ and $Z_{p,r}$ are not contained in $\Xsing$, we can define $\tilde Y$ and $\tilde Z_{p,r}$ the strict transforms of $Y$ and $Z_{p,r}$, respectively. Let $\tilde T$ be the positive closed (1,1)-current defined by

$$\tilde T:=\pi^*T\,\,\,\text{on}\,\,\,\pi^{-1}(\Xreg).$$

By assumption, $\tilde T$ has locally bounded mass around $\pi^{-1}(\Xsing)$ hence by Theorem \ref{El Mir} the extension by zero of $\tilde T$ is a positive closed (1,1)-current on $\tilde X$. On the other hand, since $\pi:\pi^{-1}(\Xreg)\to\Xreg$ is a biholomorphism we can conclude that $\nu(\tilde T,\tilde Z_{p,r})=\nu_{p,r}$ and $\nu(\tilde T,\tilde Y)=\beta$.\\

We know by Step 2 that if $\tilde\omega$ is the Fubini-Study metric on $\tilde X$ we can find a positive constant $C$ depending only on $\tilde X$, $\tilde Y$ and $\tilde\omega$ such that

$$C\geq\sum_{r\geq1}(\nu_{p,r}-\beta)^p\int_{\tilde Z_{p,r}}\tilde\omega^{\dim(X)-l-p}.$$

We prove the following lemma

\begin{lem}
Let $A$ be an ample line bundle defined on $X$ and $\tilde\omega$ the Fubini-Study metric on $\tilde X$. Then, there exist $\delta>0$ depending only on $\tilde\omega$ and $A$ such that for every irreducible algebraic set $Z\subset X$ not contained in $\Xsing$ of dimension $q$ and strict transform $\tilde Z$ the following holds

$$\int_{\tilde Z}\tilde\omega^{q}\geq \delta (A^q\cdot Z).$$
\end{lem}

\begin{proof}[Proof of Lemma]

First observe that 

$$(A^q\cdot Z)=\int_{Z}c_1(A)^{q}=\int_{\pi_*\tilde Z}c_1(A)^{q}=\int_{\tilde Z}\pi^*(c_1(A)^{q})=\int_{\tilde Z}(\pi^*c_1(A))^{q}.$$ Since $\tilde\omega$ is positive, we can find $\epsilon>$ small enough such that the class of $\{\alpha\}:=\{\tilde\omega\}-\epsilon\pi^*c_1(A)$ is numerically effective (even ample) on $\tilde X$. Then

\begin{multline*}
\int_{\tilde Z}\tilde\omega^q=\int_{\tilde Z}(\epsilon\pi^*c_1(A)+\alpha)^q=\\=\int_{\tilde Z}(\epsilon\pi^*c_1(A))^q+\sum_{i=0}^{q-1}\binom{q}{i}\int_{\tilde Z}(\epsilon\pi^*c_1(A))^i\wedge\alpha^{q-i}\geq\\\geq \epsilon^q\int_{\tilde Z}(\pi^*c_1(A))^q\geq \delta(A^q\cdot Z),
\end{multline*}
where $\delta:=\epsilon^{\dim(X)}$. This proves the lemma.

\end{proof}

Now picking $\tilde\omega$ on $\tilde X$ and $\delta>0$ as above, and taking $A=\OO_X(1)$ the theorem follows since

$$C':=C\delta^{-1}\geq \sum_{r\geq1}(\nu_{p,r}-\beta)^p\delta^{-1}\int_{\tilde Z_{p,r}}\tilde\omega^{p+l}\geq \sum_{r\geq1}(\nu_{p,r}-\beta)^p\int_{Z_{p,r}}\omega^{p+l}.$$

This completes the proof of the Main Theorem.

\newpage

%\bibliographystyle{plain}
%\bibliography{scvbib}

\def\lasp{\leavevmode\raise.45ex\hbox{$\lhook$}}

\end{document}